\documentclass{article}

\usepackage{amsmath}
\usepackage{amssymb}
\usepackage{graphicx}
\usepackage{bbm}
\usepackage{bm}
\usepackage{mathdots}
\usepackage{booktabs}
\usepackage{rotating}
\usepackage[ruled,linesnumbered]{algorithm2e}
\usepackage[a4paper,left=2.8cm,right=2.8cm,top=2.5cm,bottom=2.5cm]{geometry}
\usepackage{fancyhdr}
\usepackage{hyperref}
\usepackage{booktabs}
\usepackage{geometry}
\usepackage{listings}
\usepackage{xcolor}
\usepackage[framemethod=tikz]{mdframed}
\newtheorem{theorem}{Theorem}[section]

\makeatletter % `@' now normal "letter"
\@addtoreset{equation}{section}
\makeatother  % `@' is restored as "non-letter"

\newenvironment{proof}{{\noindent\it Proof.}\quad}{\hfill $\square$\\}
\usepackage{url}
\usepackage{mathtools}

\begin{document}
\title{The path of hyperinterpolation: A survey}

\author{Congpei An\footnotemark[1]
       \quad \text{and}\quad Jiashu Ran\footnotemark[2] \quad
       \text{and} \quad Hao-Ning Wu\footnotemark[3]}
\renewcommand{\thefootnote}{\fnsymbol{footnote}}
\footnotetext[1]{School of Mathematics and Statistics, Guizhou University, Guiyang 550025, Guizhou, China (andbachcp@gmail.com)}
\footnotetext[2]{Department of Mathematical and Statistical Sciences, University of Alberta, Edmonton, Alberta T6G 2G1, Canada (jran@ualberta.ca)}
\footnotetext[3]{Department of Mathematics, University of Georgia, Athens, GA 30602, USA (hnwu@uga.edu)}
\maketitle

\begin{abstract}
This paper surveys hyperinterpolation, a quadrature-based approximation scheme. We cover classical results, provide examples on several domains, review recent progress on relaxed quadrature exactness, introduce methodological variants, and discuss applications to differential and integral equations.
\end{abstract}



\section{Introduction}

Hyperinterpolation, a discrete analogue of the $L^2$ orthogonal projection, is a powerful technique for multivariate polynomial approximation introduced by Ian H. Sloan in \cite{sloan1995polynomial}. Let $\Omega\subset\mathbb{R}^d$ be a bounded region with a given positive measure $\mathrm{d} \omega$, which is either the closure of a connected open domain or a smooth closed lower-dimensional manifold in $\mathbb{R}^d$. This region is assumed to have finite measure with respect to $\mathrm{d} \omega$, that is, $V:=\int_{\Omega}\mathrm{d}\omega <\infty$. We wish to approximate a function $f\in C(\Omega)$ by a polynomial on $\Omega$. For lower-dimensional manifolds, this polynomial is understood as the restriction of a polynomial to the manifold, but we will simply refer to it as a ``polynomial.'' We denote by $\mathbb{P}_n \subset L^2(\Omega)$ the linear space of polynomials on $\Omega$ of degree at most $n$, equipped with the $L^2$ inner product
\begin{equation}\label{equ:innerproduct}
\langle f, g\rangle=\int_{\Omega} fg \,\mathrm{d}\omega.
\end{equation}
Let $\{p_1, p_2 \ldots, p_{d_n}\} \subset \mathbb{P}_n$ be an orthonormal basis of $\mathbb{P}_n$ in the sense of $\left\langle p_{\ell}, p_{\ell^{\prime}}\right\rangle=\delta_{\ell \ell^{\prime}}$ for $1 \leq \ell, \ell^{\prime} \leq d_n$, where $d_n=\operatorname{dim} \mathbb{P}_n$ is the dimension of $\mathbb{P}_n$, and $\delta_{\ell \ell^{\prime}}$ is the Kronecker delta. Given an $L^2$ function $f$, the $L^2$ \emph{orthogonal projection} of $f$ onto the space $\mathbb{P}_n$ is defined as
\begin{equation}\label{equ:projection}
\mathcal{P}_n f:=\sum_{\ell=1}^{d_n}\left\langle f, p_{\ell}\right\rangle p_{\ell} \in \mathbb{P}_n.
\end{equation}
However, the orthogonal projection is not computationally feasible because its coefficients involve inner products that generally cannot be evaluated exactly.

In the early 1990s, Sloan became intrigued by the conundrum of whether the interpolation of a periodic function on an interval (or equivalently, for a function on a circle) has properties as good as the orthogonal projection \eqref{equ:projection}. While interpolation on spheres and other multivariate domains remains problematic \cite{MR1845243}, a discrete approximation with the right properties offers a viable alternative, albeit one that requires more points. This approximation, now known as \emph{hyperinterpolation}, appeared in the seminal paper \cite{sloan1995polynomial}. The prefix ``hyper-'' signifies that it uses more points than standard interpolation. Sloan credited the impulsion for this new name to Werner Rheinboldt, who urged him to find an alternative to interpolation. For a brief history behind hyperinterpolation, we refer the reader to Sloan's memoir \cite{Sloan2018fourunate}.

Constructing hyperinterpolants requires an $m$-point quadrature rule of the form
\begin{equation}\label{equ:quad}
\sum_{j=1}^m w_j g(x_j) \approx \int_{\Omega} g \,\mathrm{d}\omega,
\end{equation}
where the quadrature points $x_j$ belong to $\Omega$ and weights $w_j$ are all positive for $j=1,2, \ldots, m$. Assuming that the quadrature rule \eqref{equ:quad} has exactness degree $2n$, i.e.,
\begin{equation}\label{equ:exactness}
\sum_{j=1}^m w_j g(x_j) = \int_{\Omega} g \,\mathrm{d}\omega\quad\forall g\in\mathbb{P}_{2n},
\end{equation}
the hyperinterpolation operator $\mathcal{L}_n: C(\Omega) \rightarrow \mathbb{P}_n$ maps a continuous function $f \in C(\Omega)$ to
\begin{equation}\label{equ:hyperinterpolation}
\mathcal{L}_n f:=\sum_{\ell=1}^{d_n}\left\langle f, p_{\ell}\right\rangle_m p_{\ell},
\end{equation}
where
\begin{equation}\label{equ:discreteinnerproduct}
\langle f,g\rangle_m:=\sum_{j=1}^m w_j f(x_j) g(x_j)
\end{equation}
is a ``discrete version'' of the $L^2$ inner product \eqref{equ:innerproduct}. Given $f\in C(\Omega)$, let $\mathcal{L}_nf\in\mathbb{P}_n$ be the hyperinterpolation defined by \eqref{equ:hyperinterpolation} with the quadrature rule satisfying \eqref{equ:exactness}. Sloan showed that the $\mathcal{L}_n$ is a projection operator in the sense of 
\[\mathcal{L}_nf = f\quad\forall f\in\mathbb{P}_n.\]
Moreover, the following stability and error bound for hyperinterpolation hold:
\begin{equation}\label{equ:stability}
\left\|\mathcal{L}_n f\right\|_2 \leq V^{1 / 2}\|f\|_{\infty},
\end{equation}
and
\begin{equation}\label{equ:errorbound}
\left\|\mathcal{L}_n f-f\right\|_2 \leq 2 V^{1 / 2} E_n(f),
\end{equation}
where $E_n(f):=\inf_{p\in\mathbb{P}_n}\|f-p\|_{\infty}$ denotes the best uniform approximation error of $f$. Thus, hyperinterpolation acts as a discrete analogue of the orthogonal projection \eqref{equ:projection} from $C(\Omega)$ onto $\mathbb{P}_n$. However, an $L^2\rightarrow L^2$ theory is not available in the literature.

Returning to Sloan's original insight in \cite{sloan1995polynomial}, hyperinterpolation reduces to \emph{interpolation} when the quadrature rule \eqref{equ:quad} is minimal.
An $m$-point quadrature rule \eqref{equ:quad} is said to be {minimal} if it is exact for all polynomials of degree at most $2n$ and $m = \dim \mathbb{P}_n$. Sloan showed that the classical interpolation condition
\[\mathcal{L}_nf(x_j) = f(x_j),\quad j=1,2,\ldots,m,\]
holds for arbitrary $f\in C(\Omega)$ if and only if the quadrature \eqref{equ:quad} is minimal. A famous example of a minimal rule is Gaussian quadrature on the interval. However, such rules are not always available in higher dimensions. Moreover, hyperinterpolation is closely related to the weighted \emph{least squares} approximation problem. Given $f\in C(\Omega)$, the hyperinterpolant is the minimizer of the following discrete least squares problem:
\begin{equation}\label{equ:LS}
\min _{p \in \mathbb{P}_n}~~\sum_{j=1}^m w_j\left[f(x_j)-p(x_j)\right]^2.
\end{equation}
Therefore, hyperinterpolation provides a unifying framework that connects three fundamental approximation schemes: $L^2$ orthogonal projection, interpolation, and least squares approximation. Furthermore, it is computationally straightforward to implement, making it a method that deserves significant attention.

The idea of discretizing $L^2$ orthogonal projection coefficients using quadrature rules is straightforward and appeared in the literature prior to Sloan's 1995 paper \cite{sloan1995polynomial}. For example, Atkinson and Bogomolny investigated such a discrete orthogonal projection in their work on the discrete Galerkin method for Fredholm integral equations of the second kind; see \cite[Section 4]{MR878693}. However, Sloan was the first to systematically develop it as a general-purpose approximation scheme, coining the term ``hyperinterpolation''.

\section{Hyperinterpolation on various domains}
\label{sec:domain}

Since its introduction by Sloan in  \cite{sloan1995polynomial}, hyperinterpolation has been extensively developed for approximations on the sphere \cite{dai2006generalized,MR2274179,le2001uniform,MR1761902,reimer2002generalized,MR1619076,zbMATH01421286,MR1845243}. Subsequent research has expanded its application to other regions, such as the disk \cite{hansen2009norm}, the square \cite{caliari2007hyperinterpolation}, the cube \cite{caliari2008hyperinterpolation,MR2486128,wang2014norm}, the ball \cite{MR3011254}, the spherical shell \cite{MR3554421,MR3739962}, the spherical triangle \cite{sommariva2021numerical}, the spherical polygon \cite{MR4859492}, the spherical zone \cite{chen2025spherical_zone}, and so on. In all these works, defining the hyperinterpolation operator on a domain $\Omega\subset\mathbb{R}^d$ depends on two prerequisites:
\begin{itemize}
\item the existence of a system of orthogonal polynomials on $\Omega$, and
\item the availability of a quadrature rule \eqref{equ:quad} satisfying the exactness condition \eqref{equ:exactness}.
\end{itemize}

In the rest of this section, we summarize orthogonal polynomials and quadrature rules on several domains. Due to space constraints, our summary is not exhaustive; for domains not covered here, we refer the reader to the references mentioned above.

\subsection{The interval}
For approximating functions on the weighted interval $[-1,1] \subset \mathbb{R}$, where the weight function is given by $\mathrm{d}\omega = r(x) \mathrm{d}x$ with $r\in L^1(-1,1)$, we can use the well-studied univariate orthogonal polynomials with respect to the weight function $r(x)$ as basis functions. In this case, $V=2$ and $d_n=\dim\mathbb{P}_n = n+1$.
Gauss quadrature rules are a perfect candidate for the quadrature rule \eqref{equ:quad}. An $m$-point Gauss rule is exact for polynomials of degree at most $2m-1$, not only fulfilling the exactness assumption \eqref{equ:exactness} but also making it a minimal quadrature rule. The classical reference on this case is Szeg\H{o}'s masterpiece \cite{MR0000077} on orthogonal polynomials.

\subsection{The circle}
For approximating functions on the unit circle $\mathbb{S}^1:=\{x\in\mathbb{R}^2:\|x\|_2= 1\} \subset \mathbb{R}^2$ with angular measure $\mathrm{d}\omega$ in radians, we can use the trigonometric basis $\{1, \cos \theta, \sin \theta, \ldots, \cos n \theta, \sin n \theta\}$. In this case, $V = 2\pi$ and 
$d_n=\dim\mathbb{P}_n = 2n+1$. For the quadrature rule, we employ the rectangle rule with 
$m$ equal intervals of length 
$2 \pi / m$. 
This rule is exact for all trigonometric polynomials of degree at most $\leqslant m-1$, i.e.,
\[
\frac{2 \pi}{m} \sum_{j=0}^{m-1} g\left(\frac{2 \pi j}{m}\right)=\int_0^{2 \pi} g(\theta)\, \mathrm{d} \theta \quad \forall g \in \mathbb{P}_{m-1}.
\]
Thus, the exactness condition \eqref{equ:exactness} is satisfied when $m\geq 2n+1$, also making this a minimal quadrature rule. The resulting hyperinterpolation is closely related to a result stated by Zygmund in \cite[Chapter X, Theorem 7.1]{MR107776}.

\subsection{The square}
For approximating functions on the square $[-1,1]^2\subset\mathbb{R}^2$ with the normalized product Chebyshev measure $\mathrm{d}\omega=\frac{1}{\pi^2} \frac{\mathrm{d} x_1 \mathrm{d} x_2}{\sqrt{1-x_1^2} \sqrt{1-x_2^2}}$, Caliari, De Marchi, and Vianello proposed in \cite{caliari2007hyperinterpolation} to employ tensor products of normalized Chebyshev polynomials as the basis. In this case, $V = 4$ and 
\begin{equation}\label{equ:squaredn}
d_n=\dim\mathbb{P}_n = \binom{n+2}{2} = \frac{(n+2)(n+1)}{2}.
\end{equation}
They considered quadrature rules based on Xu points \cite{MR1418495}. Starting from $n+2$ Chebyshev--Lobatto points on the interval $[-1,1]$,
\[
z_k=z_{k, n+1}=\cos \frac{k \pi}{n+1}, \quad k=0, \ldots, n+1, \quad n=2 m-1, \quad m \geq 1,
\]
the Xu points on the square $\Omega$ are defined as the two dimensional Chebyshev-like set $X_N:=A \cup B$
of cardinality $(n+1)(n+3)/2 > d_n$, where 
\[A=\left\{\left(z_{2 i}, z_{2 j+1}\right), 0 \leq i \leq m, 0 \leq j \leq m-1\right\}, \quad B=\left\{\left(z_{2 i+1}, z_{2 j}\right), 0 \leq i \leq m-1,0 \leq j \leq m\right\}.\]
The corresponding quadrature rule has exactness degree of $2n+1$, satisfying the exactness condition \eqref{equ:exactness}. However, since the cardinality of $X_N$ exceeds $d_n$, the resulting quadrature rule is not a minimal quadrature rule.

This approach using Chebyshev--Lobatto and Xu points has been extended to cubes and hypercubes \cite{caliari2008hyperinterpolation,MR2486128,wang2014norm}. A key limitation, however, is that generating a quadrature rule which satisfies the exactness condition \eqref{equ:exactness} requires more points than the dimension of $\mathbb{P}_n$. In a $d$-cube $\Omega = \{x\in\mathbb{R}^d:\|x\|_{\infty}\leq 1\}$, for example, the dimension
\[d_n = \dim\mathbb{P}_n = \binom{n+d}{d}=\frac{n^d}{d!}(1+o(1))\]
grows polynomially, yet the number of necessary quadrature points is typically much larger.

On the square $[-1,1]^2$, beginning with $n+1$ Chebyshev--Lobatto points on $[-1,1]$, one can also construct Padua points \cite{MR2271722,MR2350184,MR2457662}, a point set of size $d_n$ matching the dimension of $\mathbb{P}_n$. These points are the only known unisolvent set for bivariate interpolation with a minimally growing Lebesgue constant. However, the resulting quadrature rule is not exact on the full space $\mathbb{P}_{2n}$. Its exactness is limited to the subspace of polynomials $p(x_1,x_2)$ of degree at most $2n$ that are also orthogonal to the degree-$2n$ Chebyshev polynomial in $x_2$. This orthogonality restriction prevents it from achieving the full exactness required by \eqref{equ:exactness}. Consequently, the construction of hyperinterpolation based directly on Padua points remains an open problem.

\subsection{On the disk}
For approximating functions on the disk $\{x\in\mathbb{R}^2:\|x\|_2\leq1\}\subset\mathbb{R}^2$ with area measure $\mathrm{d}\omega$, Hansen, Atkinson, and Chien considered in \cite{hansen2009norm} using ridge polynomials \cite {MR397240} as a basis. In this case, $V= \pi$ and $d_n=\dim\mathbb{P}_n$ is the same as \eqref{equ:squaredn}.

To construct the quadrature rule, they expressed functions in radial-azimuthal coordinates as $g(x) = g(r,\phi)$. Their approach combines a trapezoidal rule in the azimuthal direction with a Gaussian quadrature rule in the radial direction, leading to the discrete inner product: 
\[\langle f, g\rangle_m :=\sum_{\ell=0}^n \sum_{m=0}^{2 n} w_\ell \frac{2}{2 n+1} r_\ell f\left(r_\ell, \frac{2 \pi m}{2 n+1}\right) g\left(r_\ell, \frac{2 \pi m}{2 n+1}\right) ,\]
where $\{w_\ell\}$ are weights of the Gauss--Legendre quadrature on $[0,1]$. This rule is exact for all polynomials $f,g\in\mathbb{P}_n$. The total number of quadrature points is $(n+1)(2n+1)$,  which exceeds the dimension $d_n$ of the polynomial space for all $n>0$, with equality holding only in the trivial case $n=0$.

\subsection{On the sphere}
For approximating functions on the unit two-sphere $\mathbb{S}^2:=\{x\in\mathbb{R}^3:\|x\|_2=1\}\subset\mathbb{R}^3$ with surface area measure ${\mathrm{d}} \omega$, we may employ spherical harmonics $\left\{Y_{\ell, k}: 0 \leq \ell \leq n, 1 \leq k \leq 2 \ell+1\right\}$ as an orthonormal basis. We refer the reader to \cite{MR0199449} for details on spherical harmonics. In this case, $V = 4\pi$ and $d_n=\dim\mathbb{P}_n =(n+1)^2$.
A natural candidate for quadrature rules satisfying the exactness assumption \eqref{equ:exactness} is the spherical $t$-design, introduced by Delsarte, Goethals, and Seidel in \cite{delsarte1991geometriae}. A point set $\left\{x_1, x_2, \ldots, x_m\right\} \subset \mathbb{S}^2$ is said to be a spherical $t$-design if it satisfies
\[
\frac{1}{m} \sum_{j=1}^m p(x_j)=\frac{1}{4 \pi} \int_{\mathbb{S}^2} p\,\mathrm{d} \omega \quad \forall p \in \mathbb{P}_t.
\]
That is, spherical $t$-design is a set of points on the sphere such that an equal-weight quadrature rule in these points integrates all (spherical) polynomials up to degree $t$ exactly. In the original manuscript \cite{delsarte1991geometriae}, lower bounds for the number $m$ of points needed to construct a spherical $t$-design were derived. It was then shown by Bannai and Damerell in \cite{MR519045,MR576179} that spherical $t$-designs of $d_n=\dim\mathbb{P}_t$ points exist only for a
few small values of $t$. In other words, despite the attractive properties of spherical $t$-esigns, their corresponding quadrature rules cannot be minimal for general $t$.

It is worth noting that spherical $t$-design  was first combined with hyperinterpolation in \cite{sloan1995polynomial}. Moreover, 
the numerical realization of hyperinterpolation by well conditioned  spherical $t$-designs \cite{an2010well} can be found in  \cite{an2011distribution, an2012regularized}. There are also other quadrature rules on the sphere $\mathbb{S}^2$, we refer the reader to the discussion in \cite[Section 4.2]{sloan1995polynomial}. 

For the $d$-sphere $\mathbb{S}^{d} := \{x \in \mathbb{R}^{d+1} : \|x\|_2 = 1\}$ with $d \geq 2$, the theory of spherical harmonics is also established. Bondarenko, Radchenko, and Viazovska \cite{Viazovska2013optimal} proved that for each $m \geq c_dt^d$, where $c_d$ is a positive constant depending only on $d$, there exists a spherical $t$-design consisting of $m$ points on $\mathbb{S}^d$. Therefore, one can conceptually define hyperinterpolation on $\mathbb{S}^{d}$. However, its practical implementation remains challenging due to the difficulty of explicitly constructing such spherical $t$-designs on high-dimensional spheres.

\section{Recent advances in relaxed quadrature exactness}
\label{sec:relaxed}

While hyperinterpolation is a powerful tool for multivariate approximation, its construction necessitates a quadrature rule with a degree of exactness $2n$. This foundational principle, however, presents significant practical and theoretical challenges. As discussed in Section \ref{sec:domain}, constructing such high-precision rules on general domains imposes stringent demands on both the number and distribution of quadrature points. In the literature on hyperinterpolation, Hesse and Sloan noted in \cite{MR2274179}: 
\begin{quote}
\emph{Hyperinterpolation is easy to compute. As a disadvantage, hyperinterpolation needs function values at the given points of the positive-weight numerical integration rule with degree of polynomial exactness $2n$ in the definition of the hyperinterpolation operator.}
\end{quote}
These practical barriers raise a deeper theoretical question: is such high exactness truly necessary for accuracy? This question also resonates with a growing scholarly trend that questions the reliability of exactness as a guide for numerical accuracy. Trefethen, for instance, argued in \cite{trefethen2022exactness} that exactness is an algebraic property concerned with exact zeros, whereas quadrature is fundamentally an analytic problem concerned with small errors.

To address this question, An and Wu proposed in \cite{an2022quadrature,an2024bypassing} relaxing the quadrature exactness assumption \eqref{equ:exactness} by Marcinkiewicz--Zygmund inequalities.

\subsection{Marcinkiewicz--Zygmund (MZ) inequalities}
The MZ inequality in the $L^2$ space states the following: Let $n\geq 0$ be an integer. There exists an absolution constant $\eta \in [0,1)$ such that 
\begin{equation}\label{equ:MZ}
\left|\sum_{j=1}^m w_j p(x_j)^2-\int_{\Omega} p^2 \,\mathrm{d} \omega\right| \leq \eta \int_{\Omega} p^2 \,\mathrm{d} \omega \quad \forall p \in \mathbb{P}_n.
\end{equation}
Since the pioneering work of Marcinkiewicz and Zygmund \cite{Marcinkiewicz1937}, there has been extensive research on establishing the MZ inequalities across various domains. In one dimension, these inequalities have been established on the torus \cite{zbMATH00539969}, the real line \cite{zbMATH01150002,zbMATH01398363}, and the unit circle \cite{zbMATH01289969}. Subsequent research extended these results to compact manifolds, with significant advances on spheres \cite{zbMATH05558973,zbMATH07844329,zbMATH05206088,mhaskar2001spherical} and general compact manifolds \cite{zbMATH05795114,filbir2011marcinkiewicz,zbMATH07227728}. In contrast, the study of the MZ inequalities on multivariate compact domains in Euclidean spaces remains relatively limited, with notable contributions for spherical caps \cite{zbMATH05676005}, preliminary results results for several domains \cite{MR3746524}, and scattered data on polygons \cite{wu2024marcinkiewicz}.

\subsection{Hyperinterpolation using exactness-relaxed quadrature rules}

To relax the strict quadrature exactness condition \eqref{equ:exactness}, An and Wu proposed in \cite{an2022quadrature,an2024bypassing} leveraging MZ inequalities \eqref{equ:MZ}. These inequalities provide a natural framework that generalizes the exactness assumption \eqref{equ:exactness}; indeed, exactness \eqref{equ:exactness} is recovered as the special case for $\eta = 0$. Their approach in \cite{an2022quadrature} reduces the required quadrature exactness from degree $2n$ to a lower degree of only $n+k$, where $0 < k \leq n$, and additionally assume the used quadrature rule \eqref{equ:quad} satisfies the MZ inequality \eqref{equ:MZ} with constant $\eta$. Then for $f\in C(\Omega)$,
\[
\left\|\mathcal{L}_n f\right\|_2 \leq \frac{V^{1 / 2}}{\sqrt{1-\eta}}\|f\|_{\infty}
\]
and
\begin{equation}\label{equ:errorbound1}
\left\|\mathcal{L}_n f-f\right\|_2 \leq\left(\frac{1}{\sqrt{1-\eta}}+1\right) V^{1 / 2} E_k(f).
\end{equation}
If the degree of quadrature exactness is not relaxed, i.e., $k=n$ and $\eta= 0$, then the approximation theory is
the same as \eqref{equ:stability} and \eqref{equ:errorbound}. If $0 < k < n$, then as a cost of the relaxation of
exactness, the error estimation \eqref{equ:errorbound1} is now controlled by $E_k(f)$ rather than $E_n(f)$. Since $E_k ( f ) \geq E_n ( f )$ if $k < n$, this estimation \eqref{equ:errorbound1} reveals an effect of relaxing the
quadrature exactness. That is, we can use fewer quadrature points than the original
hyperinterpolation, but the corresponding error estimation will be somewhat amplified. An and Wu also extended their results in \cite{an2024bypassing} to a totally relaxed case, discarding the quadrature exactness assumption and assuming the MZ ineuality \eqref{equ:MZ} only. They conducted a case study on spheres with error bounds derived. This bound not only depends on polynomial degree $n$, but also depends on the geometry of $\{x_j\}_{j=1}^m$.

\subsection{Significance of such relaxation}
Relaxing the exactness requirement offers two major practical advantages. Firstly, it enables a substantial reduction in the number of quadrature points, which is critical in high-dimensional settings where the required number grows exponentially with the dimension. Secondly, it greatly expands the range of admissible quadrature rules. This allows for the construction of hyperinterpolants with proven error bounds even from rules based on scattered data or random points with simple weights, as is often encountered in engineering practice. Although this flexibility may come at the cost of optimal approximation rates, it is essential for bridging the gap between theoretical numerical analysis and practical application.

\section{Variants of hyperinterpolation}
Since the introduction of hyperinterpolation in \cite{sloan1995polynomial}, several variants have been developed for different purposes. We briefly summarize their motivations and definitions, referring the reader to the original papers for detailed approximation theory.

\subsection{Generalized hyperinterpolation}
While initial work established a complete theory for $\mathcal{L}_n: C(\Omega) \to L^2(\Omega)$, 
the corresponding analysis for $C(\Omega) \to C(\Omega)$ remained undeveloped. 
The first approach to achieving a uniformly bounded operator norm  was through the \emph{generalized hyperinterpolation} \cite{dai2006generalized,reimer2002generalized}, 
a framework initially formulated and studied exclusively on spheres. We denote by $|\mathbb{S}^{d}|$ the surface area of $\mathbb{S}^d$ here and throughout. Let $C_{\ell}^{\frac{d-1}{2}}$ be the Gegenbauer polynomial of degree $\ell$ and of index $\frac{d-1}{2}$ on $\mathbb{S}^{d}$ and assume that
\begin{equation*}
    D_{n}(t)=\sum_{\ell=0}^{n}a_{n\ell} \frac{2\ell+d-1}{(d-1)|\mathbb{S}^d|} C_{\ell}^{\frac{d-1}{2}}(t),\quad n = 1, 2,\cdots, \quad t \in [-1,1],
\end{equation*}
is a sequence of polynomials on $[-1,1]$ satisfying $ \int_{\mathbb{S}^{d}} D_{n}({x}^{\top}{y}) \,\mathrm{d}\omega({y})=1$ for $x\in\mathbb{S}^d$ and $\sup_{n}\int_{0}^{\pi} (1+n \theta)^2 |D_n(\cos \theta)|\sin^{d-1}\theta \, \mathrm{d}\theta <\infty$.
Then the {generalized hyperinterpolation operators $\mathcal{GL}_{n}$ associated to $\{D_n(t)\}_{n=1}^{\infty}$ are defined as
    \begin{equation*}
        \mathcal{GL}_{n} f({x}) = \sum_{j=1}^{m}w_j f({x}_j)D_{n}({x}\cdot {x}_j).
    \end{equation*}
 The generalized hyperinterpolation operator defined by Reimer in \cite{reimer2002generalized} assumes the positivity of the kernels \( D_n \). Under this condition, it was proved that the approximation error can be uniformly bounded by the first-order modulus of continuity \cite{reimer2002generalized}. Later, Dai established in \cite{dai2006generalized} that the uniform approximation error of \( \mathcal{GL}_{n} \) is controlled by the second-order modulus of continuity, thereby bypassing the need for the positivity assumption on the kernels \( D_n \).

\subsection{Filtered hyperinterpolation}
To obtain the uniform convergence, Sloan and Womersley considered another approach in \cite{MR2875103,sloan2012filtered},  involving filtering techniques. Let $h$ be a continuous filter function defined as 
\begin{equation*}
h(x):=\left\{\begin{array}{ll}
1,  &\text{for}~x \in [0,1/2] ,\\
0,  & \text{for}~x\in[1,\infty).
\end{array}\right.
\end{equation*}    
Then the \emph{filtered hyperinterpolation} is defined as \cite{sloan2012filtered}
\[\mathcal{FL}_nf(x) = \sum_{j=1}^mw_jf(x_j)H_n(x_j,x), \quad \forall x\in \mathbb{S}^d,\]
with
\[H_n(x,y) = \sum_{\ell=0}^{n}h\left(\frac{\ell}{n}\right)\sum_{k=1}^{Z(d,\ell)} Y_{\ell,k}(x)Y_{\ell,k}(y), \quad \forall x,y \in \mathbb{S}^d,
\]
where 
$\{Y_{\ell,k}:\ell=0,1,\ldots,n,~k = 1,2,\ldots,Z(d,\ell)\}$ are  orthonormal spherical harmonics \cite{MR0199449}, which can be a basis for the polynomial space over $\mathbb{S}^d$. 
It can be shown that the sum of the $(d+1)$th forward difference of $\{h(\ell/n)\}_{\ell=0}^{n}$ is bounded, then the resulting operator $\mathcal{FL}_n$ was shown in \cite{sloan2012filtered} to equip a uniform operator norm from $C(\mathbb{S}^d)$ to $C(\mathbb{S}^d)$. Moreover, hyperinterpolation and filtered hyperinterpolation have a nice consistency property, with the special case on $\mathbb{S}^2$ using spherical $t$-designs revealed in \cite[Theorem 4.3.2]{an2011distribution}.
\begin{theorem}\label{thm:invariant}
%Let the orthonormal spherical harmonics $\%%{Y_{\ell,k}\}_{\ell=0,\ldots,n,k = %1,2,\ldots,Z(d,\ell)}$ be the basis.
    Given a positive-weight quadrature rule \eqref{equ:quad} with exactness $2n$ on $\mathbb{S}^d$, the exact integration of both hyperinterpolant and filtered hyperinterpolant recovers the quadrature rule used in its construction in the sense of
    \begin{equation*}
        \int_{\mathbb{S}^d} \mathcal{L}_n f({x}) \,\mathrm{d}\omega(x) = \int_{\mathbb{S}^d} \mathcal{FL}_{n} f({x}) \,\mathrm{d}\omega({x}) = \sum_{j=1}^{m} w_j f({{x}_j}).
    \end{equation*}
\end{theorem}
\begin{proof}
We first prove the result for the case of filtered hyperinterpolation. Since $\sqrt{|\mathbb{S}^{d}|} \, Y_{0,1}=1$ and $\int_{\mathbb{S}^{d}} Y_{\ell,k} ({x})\, \mathrm{d}\omega({x})=0$ for $\ell >0$, we have 
\begin{equation*}
    \begin{aligned}
        \int_{\mathbb{S}^d} \mathcal{FL}_{n} f({x}) \,\mathrm{d}\omega(x)
        &=  \sum_{\ell=0}^{n}\sum_{k=1}^{Z(d,\ell)} \sum_{j=1}^{m} h\left(\frac{\ell}{n}\right) w_j f(x_j) Y_{\ell,k}(x_j)\int_{\mathbb{S}^d}  Y_{\ell,k}(x) \,\mathrm{d}\omega(x) \\
        &= \sum_{j=1}^{m} w_j f({x}_j) \frac{1}{\sqrt{|\mathbb{S}^{d}|}} \frac{1}{\sqrt{|\mathbb{S}^{d}|}} |\mathbb{S}^{d}| = \sum_{j=1}^{m} w_j f({x_j}). 
    \end{aligned}
\end{equation*}
 The case of hyperinterpolation now follows as a special case of $h\equiv 1$.
\end{proof}
Consequently, since both $\mathcal{L}_nf$ and $\mathcal{FL}_nf$ are polynomials of degree $n$, the exactness assumption \eqref{equ:exactness} leads to
\begin{equation*}
       \sum_{j=1}^{m} w_j\mathcal{L}_n f(x_j) =  \sum_{j=1}^{m} w_j\mathcal{FL}_{n} f(x_j) =  \sum_{j=1}^{m} w_j f({x_j}).
    \end{equation*}
    Yet we note that in general, $\mathcal{L}_nf$, $\mathcal{FL}_nf$, and $f$ have different values at the quadrature points.

The scheme of filtered hyperinterpolation has also been extended to general manifolds and the setting of distributed learning in \cite{MR4226998,MR4462410}. We note that Mhaskar also proposed in \cite{Mhaskar2005sphere,Mhaskar2006weighted} a fully discrete filtered polynomial approximation scheme that is equivalent to filtered hyperinterpolation.

\subsection{Regularized hyperinterpolation}
If samples $\{f(x_j)\}_{j=1}^m$ in the construction of hyperinterpolation are contaminated by noise, we may consider \emph{regularized hyperinterpolation}. This is achieved by adding an explicit regularization term $\mathcal{R}(p)$ onto the least squares approximation problem, resulting the regularized least squares approximation problem 
\begin{equation}\label{equ:regularizedLS}
\min _{p \in \mathbb{P}_n}~~\left\{\sum_{j=1}^m w_j\left[f(x_j)-p(x_j)\right]^2+\lambda\mathcal{R}(p)\right\}, \qquad \lambda >0.
\end{equation}
Let 
$$p(x) = \sum_{\ell=1}^{d_n}\beta_{\ell}p_{\ell}(x).$$
The following are several regularized hyperinterpolants:
\begin{itemize}
\item If $\mathcal{R}(p) = \sum_{\ell=1}^{d_n}|\beta_{\ell}|^2$, then we have the Tikhonov or $\ell_2$-regularized hyperinterpolation \cite{an2012regularized,an2020tikhonov}:
\[p(x) = \frac{1}{1+\lambda}\sum_{\ell=1}^{d_n}\langle f,p_{\ell}\rangle_m p_{\ell}(x).\]
\item If $\mathcal{R}(p) = 2\sum_{\ell=1}^{d_n}|\beta_{\ell}|$, then we have the Lasso or $\ell_1$-regularized hyperinterpolation \cite{MR4644764,an2021lasso}:
\[p(x) = \sum_{\ell=1}^{d_n}S_{\lambda}^{\text{soft}}\left(\langle f,p_{\ell}\rangle_m\right) p_{\ell}(x),\]
where 
\begin{equation*}
S_{\lambda}^{\text{soft}} (\alpha) = \begin{cases}
(|\alpha|-\lambda)\text{sign}(\alpha), & |\alpha|>\lambda,\\
0, & |\alpha|\leq \lambda,
\end{cases}
\end{equation*}
is the soft shrinkage operator \cite{donoho1994ideal}.
\item If $\mathcal{R}(p) = \sum_{\ell=1}^{d_n}\mathbbm{1}_{\beta_{\ell} \neq 0}$, then we have the hard thresholding or $\ell_0$-regularized hyperinterpolation \cite{an2025hard}:
\begin{equation*}
    p(x) = \sum_{\ell=1}^{d_n} S_{\lambda}^{\text{hard}}(\langle f, p_{\ell} \rangle_m)p_{\ell}(x),
\end{equation*}
where
\begin{equation*}
S_{\lambda}^{\text{hard}}(\alpha) = \begin{cases}
\alpha, & |\alpha|>\sqrt{\lambda},\\
0, & |\alpha|\leq \sqrt{\lambda},
\end{cases}
\end{equation*}
is the hard thresholding operator \cite{donoho1994ideal}.
\end{itemize}
Combinations of these regularization terms give rise to novel regularization schemes, such as hybrid hyperinterpolation \cite{an2025hybrid}, springback hyperinterpolation \cite{zbMATH07576449}, and recovery thresholding hyperinterpolation \cite{an2025recovery}.

We emphasize that these regularized hyperinterpolation variants likewise require the quadrature exactness assumption \eqref{equ:exactness}. This is a key assumption to derive closed-form solutions to the regularized least squares approximation problem \eqref{equ:regularizedLS}. The integration of recent developments in relaxed exactness, discussed in Section \ref{sec:relaxed}, presents an open research direction.

\subsection{Efficient hyperinterpolation}
If the function $f$ to be approximated contains an oscillatory or singular kernel in the form of 
\[f(x) = K(x)F(x),\]
where $K\in L^1(\Omega)$ is a real- or complex-valued absolutely
integrable function encoding oscillatory or singular behavior, which needs not be continuous or of one sign, and $F(x)\in C(\Omega)$ is a
continuous (and preferably smooth) function, An and Wu proposed the \emph{efficient hyperinterpolation} in \cite{an2024hyperinterpolation}.

It is well known that in the presence of a singular or highly oscillatory kernel $K$, it is inefficient to
evaluate the coefficients $\langle KF, p_{\ell}\rangle$ directly using some classical numerical integration rules. Instead, one shall evaluate them in a semi-analytical way: for the evaluation of an integral of the form $\int_{\Omega} K(x)F(x)\,\mathrm{d}\omega(x)$,
one shall replace $F$ by its polynomial interpolant or approximant of degree $n$ in terms of $F\approx\sum_{\ell=1}^{d_n}c_{\ell}p_{\ell}$, and evaluate the integral by

\begin{equation*}
\int_{\Omega} K(x)F(x)\,\mathrm{d}\omega(x)\approx \sum_{\ell=1}^{d_n}c_{\ell}\int_{\Omega} K(x)p_{\ell}(x)\,\mathrm{d}\omega(x).
\end{equation*}
This idea for numerical integration may be referred to as the \emph{product-integration rule} in the classical literature \cite{MR494863,MR577405,MR650061}. This rule was initially designed on $[-1,1]$ for $K\in L^1[-1,1]$ and $F\in{C}[-1,1]$, and it converges to the exact integral as the number of quadrature points approaches the infinity if $K\in L^p[-1,1]$ for some $p>1$ is additionally assumed. In the context of highly oscillatory integrals with an oscillatory $K\in {C}(\Omega)$, this approach is also known as the \emph{Filon-type method} \cite{zbMATH02573604,MR2211043}. In most of these references, $f$ is approximated by its interpolant, and it is generally assumed that the \emph{modified moments}

\begin{equation}\label{equ:modifiedmoments}
\int_{\Omega} K(x)p_{\ell}(x)\,\mathrm{d}\omega(x),\quad \ell = 1,2,\ldots,d_n,
\end{equation}
can be computed analytically by using special functions or efficiently by invoking some stable iterative procedures.

Therefore, the efficient hyperinterpolation takes the form of 
\begin{equation*}\begin{split}
    \mathcal{EL}_n f:=&\sum_{\ell=1}^{d_n}\left(\int_{\Omega} K\left(\mathcal{L}_n F\right) p_{\ell} \,\mathrm{d} \omega\right) p_{\ell} =\int_{\Omega} K\left[\sum_{\ell^{\prime}=1}^{d_n}\left(\sum_{j=1}^m w_j F(x_j) p_{\ell^{\prime}}(x_j)\right) p_{\ell^{\prime}}\right] p_{\ell} \,\mathrm{d} \omega \\
    =&\sum_{j=1}^m w_j\left(\sum_{\ell^{\prime}=1}^{d_n} p_{\ell^{\prime}}(x_j) \int_{\Omega} K p_{\ell^{\prime}} p_{\ell} \,\mathrm{d} \omega\right) F(x_j),
    \end{split}\end{equation*}
where $\int_{\Omega} K p_{\ell^{\prime}} p_{\ell} \,\mathrm{d} \omega$ can be efficiently evaluated via the modified moments \eqref{equ:modifiedmoments} by first expressing each product $p_{\ell^{\prime}} p_{\ell}$ in the orthonormal basis.

\section{Beyond function approximation}

While most literature on hyperinterpolation focuses on approximation properties, the method has also been extended to solve differential and integral equations. The discrete nature of hyperinterpolation, which is built upon quadrature rules, automatically incorporates quadrature directly into the numerical analysis of hyperinterpolation-based methods.

\subsection{Partial differential equations}
Numerical methods for approximating differential equations are predominantly based on either collocation or Galerkin principles. The focus of hyperinterpolation-based methods has been on Galerkin-type methods, as collocation approaches rely on a mature theory of multivariate polynomial interpolation, a theory that remains underdeveloped compared to its one-dimensional counterpart and thus presents significant theoretical challenges.

Consider an elliptical equation
\begin{equation}
    \label{equ:Lu}
    Lu = f
\end{equation}
subject to appropriate boundary conditions, where $L$ is a linear differential operator. Applying the degree-$n$ hyperinterpolation operator to both sides of \eqref{equ:Lu} yields $\mathcal{L}_n(Lu) = \mathcal{L}_nf$.
Expanding this using the orthogonal polynomial basis $\{p_{\ell}\}_{\ell=1}^{d_n}$ gives
\[\sum_{\ell=1}^{d_n}\langle Lu,p_{\ell}\rangle_m p_{\ell} = \sum_{\ell=1}^{d_n}\langle f,p_{\ell}\rangle_m p_{\ell},\]
where $\langle\cdot,\cdot\rangle_m$ denotes the discrete inner product \eqref{equ:discreteinnerproduct}. By the orthogonality of the basis, this simplifies to a system of discrete equations
\[\langle Lu,p_{\ell}\rangle_m = \langle f,p_{\ell}\rangle_m,\quad \ell = 1,2,\ldots,d_n,\]
which is a discrete version of the Galerkin scheme
$\langle Lu,p_{\ell}\rangle = \langle f,p_{\ell}\rangle$ for all $\ell = 1,2,\ldots,d_n$. For the discussion on hyperinterpolation on discrete Galerkin schemes as described above, we refer the reader to \cite{MR2659585,MR2776446,Gross2018Galerkin_hyperinterpolation,Gross2018spectral_hyperinterpolation,hansen2009norm}. 

Hyperinterpolation can also be extended to time-dependent problems. Consider the semilinear equation
\begin{equation}
\label{eq:semilinear}
\frac{\partial u(x,t)}{\partial t} = Lu(x,t) + f(x, u(x,t))
\end{equation}
equipped with suitable initial and boundary conditions, where $L$ is a linear differential operator and $ f(x,u(x,t))$ is a nonlinear terms. In this context, hyperinterpolation serves as a tool to handle the nonlinearity. By applying the hyperinterpolation operator to the nonlinear term $ f(x, u(x,t))$, it is effectively linearized at the discrete level, facilitating the construction of a spectral method for equations of the type \eqref{eq:semilinear}. As a representative example, the Allen--Cahn equation on spheres has been solved using this methodology in \cite{wu2023breaking}.

Hyperinterpolation provides a framework that inherently incorporates quadrature into the numerical analysis of methods derived from it. Even when quadrature rules fail to satisfy the exactness assumption \eqref{equ:exactness}, one can still define a hyperinterpolation operator and a corresponding numerical scheme based on the available discrete samples. To the best of the authors' knowledge, \cite{wu2023breaking} presents the only work to date that investigates hyperinterpolation-based methods for partial differential equations using such relaxed quadrature conditions, successfully integrating the theoretical results of Section \ref{sec:relaxed} into its numerical analysis.

\subsection{Integral equations}
The application of hyperinterpolation extends to the numerical solution of integral equations. Indeed, the approximation scheme was first conceived in \cite{MR878693} as a discrete orthogonal projection specifically for solving Fredholm integral equations of the second kind
\begin{equation}\label{equ:integral}
u({x})-\int_{\mathbb{S}^2} K({x}, {y}) u({y})\, \mathrm{d} \omega({y})=f({x}),
\end{equation}
which can be written symbolically as $(I-\mathcal{K})u = f$.

There are two main techniques in which hyperinterpolation plays a role. The first arises in projection-based methods for integral equations; see the original manuscript \cite{MR878693} by Atkinson and Bogomolny. By replacing the orthogonal projection operator with the hyperinterpolation operator, one obtains the discrete system:
\begin{equation*}
(I-\mathcal{L}_n\mathcal{K})u_n = \mathcal{L}_nf,
\end{equation*}
where $u_n\in\mathbb{P}_n$ is the numerical solution.

The second technique concerns kernels with singularities. When the kernel $K$ admits a factorization $K=K_1K_2$, where $K_1$ is singular and $K_2$ is continuous, a specialized technique can be applied. Similar to the principles of efficient hyperinterpolation, if the modified moments \eqref{equ:modifiedmoments} for the singular part $K_1$ are computable, the continuous component $K_2u$ in the integral operator can be replaced by its hyperinterpolant. This methodology, often combined with the first projection-based technique, is explored in \cite{MR2071394,MR3822242,MR1922922}. A key insight for its use with the singular kernel $1/|x-y|$ on spheres originated from Wienert in \cite{Wienert1990}, facilitating fast algorithms for direct and inverse acoustic scattering problems in $\mathbb{R}^3$.

Beyond projection-based methods, if one solely considers the second technique and selects a set of collocation points, the result is a quadrature-based method, also known as a Nystr\"{o}m method. This approach has been realized in a recent work \cite{an2024hyperinterpolation}, which builds on the results for relaxed quadrature rules from Section \ref{sec:relaxed}. A compelling open question is how to incorporate the results of Section \ref{sec:relaxed} into projection-based methods, thereby broadening their applicability to a wider range of quadrature rules that may not satisfy the restrictive exactness assumption \eqref{equ:exactness}.

\section{Concluding remarks}

In summary, hyperinterpolation has evolved from a discrete analogue of $L^2$ projection into a versatile framework for multivariate approximation. Its core elegance has inspired two key advances: the relaxation of the quadrature exactness assumption \eqref{equ:exactness} via Marcinkiewicz--Zygmund inequalities, and the creation of specialized variants for convergence, noise, and singularities. Despite its significance and practical utility, the term ``hyperinterpolation'' remains conspicuously absent from the indices of most multivariate approximation texts, thirty years after its introduction in 1995. We hope this will be different a generation from now.
\section{Acknowledgement}
The work was supported by the National Natural Science Foundation of China (Project No. 12371099).

\newpage
\bibliographystyle{siamplain}
\bibliography{myref}

\begin{thebibliography}{10}

\bibitem{an2011distribution}
{\sc C.~An}, {\em Distribution of points on the sphere and spherical designs},
  PhD thesis, The Hong Kong Polytechnic University, 2011.

\bibitem{MR4644764}
{\sc C.~An and M.~Cai}, {\em Lasso trigonometric polynomial approximation for
  periodic function recovery in equidistant points}, Appl. Numer. Math., 194
  (2023), pp.~115--130,
  \href{http://dx.doi.org/10.1016/j.apnum.2023.09.001}{doi:\nolinkurl{10.1016/j.apnum.2023.09.001}}.

\bibitem{an2010well}
{\sc C.~An, X.~Chen, I.~H. Sloan, and R.~S. Womersley}, {\em Well conditioned
  spherical designs for integration and interpolation on the two-sphere}, SIAM
  J. Numer. Anal., 48 (2010), pp.~2135--2157.

\bibitem{an2012regularized}
{\sc C.~An, X.~Chen, I.~H. Sloan, and R.~S. Womersley}, {\em Regularized least
  squares approximations on the sphere using spherical designs}, SIAM J. Numer.
  Anal., 50 (2012), pp.~1513--1534,
  \href{http://dx.doi.org/10.1137/110838601}{doi:\nolinkurl{10.1137/110838601}}.

\bibitem{an2025hard}
{\sc C.~An and J.~Ran}, {\em Hard thresholding hyperinterpolation over general
  regions}, J. Sci. Comput., 102 (2025), pp.~1--26,
  \href{http://dx.doi.org/10.1007/s10915-024-02754-4}{doi:\nolinkurl{10.1007/s10915-024-02754-4}}.

\bibitem{an2025recovery}
{\sc C.~An and J.~Ran}, {\em Recovery thresholding hyperinterpolations in
  signal processing}, arXiv preprint arXiv: 2507.17916,  (2025).

\bibitem{an2025hybrid}
{\sc C.~An, J.~Ran, and A.~Sommariva}, {\em Hybrid hyperinterpolation over
  general regions}, Calcolo, 62 (2025), pp.~1--23,
  \href{http://dx.doi.org/10.1007/s10092-024-00625-w}{doi:\nolinkurl{10.1007/s10092-024-00625-w}}.

\bibitem{an2021lasso}
{\sc C.~An and H.-N. Wu}, {\em Lasso hyperinterpolation over general regions},
  SIAM J. Sci. Comput., 43 (2021), pp.~A3967--A3991,
  \href{http://dx.doi.org/10.1137/20M137793X}{doi:\nolinkurl{10.1137/20M137793X}}.

\bibitem{an2020tikhonov}
{\sc C.~An and H.-N. Wu}, {\em Tikhonov regularization for polynomial
  approximation problems in {G}auss quadrature points}, Inverse Problems, 37
  (2021), 015008 (19~pages),
  \href{http://dx.doi.org/10.1088/1361-6420/abcd44}{doi:\nolinkurl{10.1088/1361-6420/abcd44}}.

\bibitem{an2022quadrature}
{\sc C.~An and H.-N. Wu}, {\em On the quadrature exactness in
  hyperinterpolation}, BIT, 62 (2022), pp.~1899--1919,
  \href{http://dx.doi.org/10.1007/s10543-022-00935-x}{doi:\nolinkurl{10.1007/s10543-022-00935-x}}.

\bibitem{an2024bypassing}
{\sc C.~An and H.-N. Wu}, {\em Bypassing the quadrature exactness assumption of
  hyperinterpolation on the sphere}, J. Complexity, 80 (2024), p.~101789,
  \href{http://dx.doi.org/10.1016/j.jco.2023.101789}{doi:\nolinkurl{10.1016/j.jco.2023.101789}}.

\bibitem{an2024hyperinterpolation}
{\sc C.~An and H.-N. Wu}, {\em Is hyperinterpolation efficient in the
  approximation of singular and oscillatory functions?}, J. Approx. Theory, 299
  (2024), p.~106013,
  \href{http://dx.doi.org/10.1016/j.jat.2023.106013}{doi:\nolinkurl{10.1016/j.jat.2023.106013}}.

\bibitem{zbMATH07576449}
{\sc C.~An, H.-N. Wu, and X.~Yuan}, {\em The springback penalty for robust
  signal recovery}, Appl. Comput. Harmon. Anal., 61 (2022), pp.~319--346,
  \href{http://dx.doi.org/10.1016/j.acha.2022.07.002}{doi:\nolinkurl{10.1016/j.acha.2022.07.002}}.

\bibitem{MR878693}
{\sc K.~Atkinson and A.~Bogomolny}, {\em The discrete {G}alerkin method for
  integral equations}, Math. Comp., 48 (1987), pp.~595--616, S11--S15,
  \href{http://dx.doi.org/10.2307/2007830}{doi:\nolinkurl{10.2307/2007830}}.

\bibitem{MR2659585}
{\sc K.~Atkinson, D.~Chien, and O.~Hansen}, {\em A spectral method for elliptic
  equations: the {D}irichlet problem}, Adv. Comput. Math., 33 (2010),
  pp.~169--189,
  \href{http://dx.doi.org/10.1007/s10444-009-9125-8}{doi:\nolinkurl{10.1007/s10444-009-9125-8}}.

\bibitem{MR2776446}
{\sc K.~Atkinson, O.~Hansen, and D.~Chien}, {\em A spectral method for elliptic
  equations: the {N}eumann problem}, Adv. Comput. Math., 34 (2011),
  pp.~295--317,
  \href{http://dx.doi.org/10.1007/s10444-010-9154-3}{doi:\nolinkurl{10.1007/s10444-010-9154-3}}.

\bibitem{MR519045}
{\sc E.~Bannai and R.~M. Damerell}, {\em Tight spherical designs. {I}}, J.
  Math. Soc. Japan, 31 (1979), pp.~199--207,
  \href{http://dx.doi.org/10.2969/jmsj/03110199}{doi:\nolinkurl{10.2969/jmsj/03110199}}.

\bibitem{MR576179}
{\sc E.~Bannai and R.~M. Damerell}, {\em Tight spherical designs. {II}}, J.
  London Math. Soc. (2), 21 (1980), pp.~13--30,
  \href{http://dx.doi.org/10.1112/jlms/s2-21.1.13}{doi:\nolinkurl{10.1112/jlms/s2-21.1.13}}.

\bibitem{Viazovska2013optimal}
{\sc A.~Bondarenko, D.~Radchenko, and M.~Viazovska}, {\em Optimal asymptotic
  bounds for spherical designs}, Ann. Math., 178 (2013), pp.~443--452,
  \href{http://dx.doi.org/10.4007/annals.2013.178.2.2}{doi:\nolinkurl{10.4007/annals.2013.178.2.2}}.

\bibitem{MR2271722}
{\sc L.~Bos, M.~Caliari, S.~De~Marchi, M.~Vianello, and Y.~Xu}, {\em Bivariate
  {L}agrange interpolation at the {P}adua points: the generating curve
  approach}, J. Approx. Theory, 143 (2006), pp.~15--25,
  \href{http://dx.doi.org/10.1016/j.jat.2006.03.008}{doi:\nolinkurl{10.1016/j.jat.2006.03.008}}.

\bibitem{MR2350184}
{\sc L.~Bos, S.~De~Marchi, M.~Vianello, and Y.~Xu}, {\em Bivariate {L}agrange
  interpolation at the {P}adua points: the ideal theory approach}, Numer.
  Math., 108 (2007), pp.~43--57,
  \href{http://dx.doi.org/10.1007/s00211-007-0112-z}{doi:\nolinkurl{10.1007/s00211-007-0112-z}}.

\bibitem{zbMATH05558973}
{\sc A.~B{\"o}ttcher, S.~Kunis, and D.~Potts}, {\em Probabilistic spherical
  {Marcinkiewicz}-{Zygmund} inequalities}, J. Approx. Theory, 157 (2009),
  pp.~113--126,
  \href{http://dx.doi.org/10.1016/j.jat.2008.07.006}{doi:\nolinkurl{10.1016/j.jat.2008.07.006}}.

\bibitem{caliari2007hyperinterpolation}
{\sc M.~Caliari, S.~De~Marchi, and M.~Vianello}, {\em Hyperinterpolation on the
  square}, J. Comput. Appl. Math., 210 (2007), pp.~78--83,
  \href{http://dx.doi.org/10.1016/j.cam.2006.10.058}{doi:\nolinkurl{10.1016/j.cam.2006.10.058}}.

\bibitem{MR2457662}
{\sc M.~Caliari, S.~De~Marchi, and M.~Vianello}, {\em Bivariate {L}agrange
  interpolation at the {P}adua points: {C}omputational aspects}, J. Comput.
  Appl. Math., 221 (2008), pp.~284--292,
  \href{http://dx.doi.org/10.1016/j.cam.2007.10.027}{doi:\nolinkurl{10.1016/j.cam.2007.10.027}}.

\bibitem{caliari2008hyperinterpolation}
{\sc M.~Caliari, S.~De~Marchi, and M.~Vianello}, {\em Hyperinterpolation in the
  cube}, Comput. Math. Appl., 55 (2008), pp.~2490--2497,
  \href{http://dx.doi.org/10.1016/j.camwa.2007.10.003}{doi:\nolinkurl{10.1016/j.camwa.2007.10.003}}.

\bibitem{zbMATH01289969}
{\sc C.~K. Chui and L.~Zhong}, {\em Polynomial interpolation and
  {Marcinkiewicz}-{Zygmund} inequalities on the unit circle}, J. Math. Anal.
  Appl., 233 (1999), pp.~387--405,
  \href{http://dx.doi.org/10.1006/jmaa.1999.6337}{doi:\nolinkurl{10.1006/jmaa.1999.6337}}.

\bibitem{dai2006generalized}
{\sc F.~Dai}, {\em On generalized hyperinterpolation on the sphere}, Proc.
  Amer. Math. Soc., 134 (2006), pp.~2931--2941,
  \href{http://dx.doi.org/10.1090/S0002-9939-06-08421-8}{doi:\nolinkurl{10.1090/S0002-9939-06-08421-8}}.

\bibitem{zbMATH05676005}
{\sc F.~Dai and H.~Wang}, {\em Positive cubature formulas and
  {Marcinkiewicz}-{Zygmund} inequalities on spherical caps}, Constr. Approx.,
  31 (2010), pp.~1--36,
  \href{http://dx.doi.org/10.1007/s00365-009-9041-7}{doi:\nolinkurl{10.1007/s00365-009-9041-7}}.

\bibitem{MR3746524}
{\sc S.~De~Marchi and A.~Kro\'{o}}, {\em Marcinkiewicz-{Z}ygmund type results
  in multivariate domains}, Acta Math. Hungar., 154 (2018), pp.~69--89,
  \href{http://dx.doi.org/10.1007/s10474-017-0769-4}{doi:\nolinkurl{10.1007/s10474-017-0769-4}}.

\bibitem{MR2486128}
{\sc S.~De~Marchi, M.~Vianello, and Y.~Xu}, {\em New cubature formulae and
  hyperinterpolation in three variables}, BIT, 49 (2009), pp.~55--73,
  \href{http://dx.doi.org/10.1007/s10543-009-0210-7}{doi:\nolinkurl{10.1007/s10543-009-0210-7}}.

\bibitem{delsarte1991geometriae}
{\sc P.~Delsarte, J.-M. Goethals, and J.~J. Seidel}, {\em Spherical codes and
  designs}, Geom. Dedicata, 6 (1977), pp.~363--388,
  \href{http://dx.doi.org/10.1007/bf03187604}{doi:\nolinkurl{10.1007/bf03187604}}.

\bibitem{donoho1994ideal}
{\sc D.~L. Donoho and I.~M. Johnstone}, {\em {Ideal spatial adaptation by
  wavelet shrinkage}}, Biometrika, 81 (1994), pp.~425--455,
  \href{http://dx.doi.org/10.1093/biomet/81.3.425}{doi:\nolinkurl{10.1093/biomet/81.3.425}}.

\bibitem{zbMATH07844329}
{\sc F.~Filbir, R.~Hielscher, T.~Jahn, and T.~Ullrich}, {\em
  Marcinkiewicz-{Zygmund} inequalities for scattered and random data on the
  $q$-sphere}, Appl. Comput. Harmon. Anal., 71 (2024), p.~No. 101651,
  \href{http://dx.doi.org/10.1016/j.acha.2024.101651}{doi:\nolinkurl{10.1016/j.acha.2024.101651}}.

\bibitem{zbMATH05795114}
{\sc F.~Filbir and H.~N. Mhaskar}, {\em A quadrature formula for diffusion
  polynomials corresponding to a generalized heat kernel}, J. Fourier Anal.
  Appl., 16 (2010), pp.~629--657,
  \href{http://dx.doi.org/10.1007/s00041-010-9119-4}{doi:\nolinkurl{10.1007/s00041-010-9119-4}}.

\bibitem{filbir2011marcinkiewicz}
{\sc F.~Filbir and H.~N. Mhaskar}, {\em Marcinkiewicz--{Z}ygmund measures on
  manifolds}, J. Complexity, 27 (2011), pp.~568--596,
  \href{http://dx.doi.org/10.1016/j.jco.2011.03.002}{doi:\nolinkurl{10.1016/j.jco.2011.03.002}}.

\bibitem{zbMATH02573604}
{\sc L.~N.~G. Filon}, {\em On a quadrature formula for trigonometric
  integrals}, Proc. R. Soc. Edinburgh, 49 (1929), pp.~38--47,
  \href{http://dx.doi.org/10.1017/S0370164600026262}{doi:\nolinkurl{10.1017/S0370164600026262}}.

\bibitem{MR2071394}
{\sc M.~Ganesh and I.~G. Graham}, {\em A high-order algorithm for obstacle
  scattering in three dimensions}, J. Comput. Phys., 198 (2004), pp.~211--242,
  \href{http://dx.doi.org/10.1016/j.jcp.2004.01.007}{doi:\nolinkurl{10.1016/j.jcp.2004.01.007}}.

\bibitem{MR3822242}
{\sc M.~Ganesh and S.~C. Hawkins}, {\em Hyperinterpolation for spectral wave
  propagation models in three dimensions}, in Contemporary Computational
  Mathematics - A Celebration of the 80th Birthday of Ian Sloan, Springer,
  Cham, 2018, pp.~351--372.

\bibitem{MR1922922}
{\sc I.~G. Graham and I.~H. Sloan}, {\em Fully discrete spectral boundary
  integral methods for {H}elmholtz problems on smooth closed surfaces in {$\Bbb
  R^3$}}, Numer. Math., 92 (2002), pp.~289--323,
  \href{http://dx.doi.org/10.1007/s002110100343}{doi:\nolinkurl{10.1007/s002110100343}}.

\bibitem{zbMATH00539969}
{\sc K.~Gr{\"o}chenig}, {\em A discrete theory of irregular sampling}, Linear
  Algebra Appl., 193 (1993), pp.~129--150,
  \href{http://dx.doi.org/10.1016/0024-3795(93)90275-S}{doi:\nolinkurl{10.1016/0024-3795(93)90275-S}}.

\bibitem{zbMATH07227728}
{\sc K.~Gr{\"o}chenig}, {\em Sampling, {Marcinkiewicz}-{Zygmund} inequalities,
  approximation, and quadrature rules}, J. Approx. Theory, 257 (2020), p.~No.
  105455,
  \href{http://dx.doi.org/10.1016/j.jat.2020.105455}{doi:\nolinkurl{10.1016/j.jat.2020.105455}}.

\bibitem{Gross2018Galerkin_hyperinterpolation}
{\sc B.~Gross and P.~Atzberger}, {\em Hydrodynamic flows on curved surfaces:
  Spectral numerical methods for radial manifold shapes}, J. Comput. Phys., 371
  (2018), pp.~663--689,
  \href{http://dx.doi.org/10.1016/j.jcp.2018.06.013}{doi:\nolinkurl{10.1016/j.jcp.2018.06.013}}.

\bibitem{Gross2018spectral_hyperinterpolation}
{\sc B.~Gross and P.~Atzberger}, {\em Spectral numerical exterior calculus
  methods for differential equations on radial manifolds}, J. Sci. Comput., 76
  (2018), pp.~145--165,
  \href{http://dx.doi.org/10.1007/s10915-017-0617-2}{doi:\nolinkurl{10.1007/s10915-017-0617-2}}.

\bibitem{hansen2009norm}
{\sc O.~Hansen, K.~Atkinson, and D.~Chien}, {\em On the norm of the
  hyperinterpolation operator on the unit disc and its use for the solution of
  the nonlinear poisson equation}, IMA J. Numer. Anal., 29 (2009),
  pp.~257--283,
  \href{http://dx.doi.org/10.1093/imanum/drm052}{doi:\nolinkurl{10.1093/imanum/drm052}}.

\bibitem{MR2274179}
{\sc K.~Hesse and I.~H. Sloan}, {\em Hyperinterpolation on the sphere}, in
  Frontiers in Interpolation and Approximation, vol.~282 of Pure Appl. Math.
  (Boca Raton), Chapman \& Hall/CRC, Boca Raton, 2007, pp.~213--248.

\bibitem{MR2211043}
{\sc A.~Iserles and S.~P. N{\o}rsett}, {\em On quadrature methods for highly
  oscillatory integrals and their implementation}, BIT, 44 (2004),
  pp.~755--772,
  \href{http://dx.doi.org/10.1007/s10543-004-5243-3}{doi:\nolinkurl{10.1007/s10543-004-5243-3}}.

\bibitem{MR3554421}
{\sc Y.~Kazashi}, {\em A fully discretised polynomial approximation on
  spherical shells}, GEM Int. J. Geomath., 7 (2016), pp.~299--323,
  \href{http://dx.doi.org/10.1007/s13137-016-0084-1}{doi:\nolinkurl{10.1007/s13137-016-0084-1}}.

\bibitem{MR3739962}
{\sc Y.~Kazashi}, {\em A fully discretised filtered polynomial approximation on
  spherical shells}, J. Comput. Appl. Math., 333 (2018), pp.~428--441,
  \href{http://dx.doi.org/10.1016/j.cam.2017.11.005}{doi:\nolinkurl{10.1016/j.cam.2017.11.005}}.

\bibitem{le2001uniform}
{\sc Q.~T. Le~Gia and I.~H. Sloan}, {\em The uniform norm of hyperinterpolation
  on the unit sphere in an arbitrary number of dimensions}, Constr. Approx., 17
  (2001), pp.~249--265,
  \href{http://dx.doi.org/10.1007/s003650010025}{doi:\nolinkurl{10.1007/s003650010025}}.

\bibitem{chen2025spherical_zone}
{\sc C.~Li and X.~Chen}, {\em Spherical zone $t$-designs for numerical
  integration and approximation}, SIAM J. Numer. Anal., 63 (2025),
  pp.~2072--2093,
  \href{http://dx.doi.org/10.1137/24M1718883}{doi:\nolinkurl{10.1137/24M1718883}}.

\bibitem{MR4226998}
{\sc S.-B. Lin, Y.~G. Wang, and D.-X. Zhou}, {\em Distributed filtered
  hyperinterpolation for noisy data on the sphere}, SIAM J. Numer. Anal., 59
  (2021), pp.~634--659,
  \href{http://dx.doi.org/10.1137/19M1281095}{doi:\nolinkurl{10.1137/19M1281095}}.

\bibitem{MR397240}
{\sc B.~F. Logan and L.~A. Shepp}, {\em Optimal reconstruction of a function
  from its projections}, Duke Math. J., 42 (1975), pp.~645--659.

\bibitem{zbMATH01150002}
{\sc D.~S. Lubinsky}, {\em Marcinkiewicz-{Zygmund} inequalities: {M}ethods and
  results}, in Recent progress in inequalities. Dedicated to Prof. Dragoslav S.
  Mitrinovi\'c, Kluwer Academic Publishers, Dordrecht, 1998, pp.~213--240.

\bibitem{zbMATH01398363}
{\sc D.~S. Lubinsky}, {\em On converse {Marcinkiewicz}-{Zygmund} inequalities
  in {{\(L_p,p>1\)}}}, Constr. Approx., 15 (1999), pp.~577--610,
  \href{http://dx.doi.org/10.1007/s003659900123}{doi:\nolinkurl{10.1007/s003659900123}}.

\bibitem{Marcinkiewicz1937}
{\sc J.~Marcinkiewicz and A.~Zygmund}, {\em Sur les fonctions
  ind\'{e}pendantes}, Fund. Math., 29 (1937), pp.~60--90,
  \url{http://eudml.org/doc/212925}.

\bibitem{zbMATH05206088}
{\sc J.~Marzo}, {\em Marcinkiewicz-{Zygmund} inequalities and interpolation by
  spherical harmonics}, J. Funct. Anal., 250 (2007), pp.~559--587,
  \href{http://dx.doi.org/10.1016/j.jfa.2007.05.010}{doi:\nolinkurl{10.1016/j.jfa.2007.05.010}}.

\bibitem{Mhaskar2005sphere}
{\sc H.~N. Mhaskar}, {\em {On the representation of smooth functions on the
  sphere using finitely many bits}}, Appl. Comput. Harmon. Anal., 18 (2005),
  pp.~215--233,
  \href{http://dx.doi.org/10.1016/j.acha.2004.11.004}{doi:\nolinkurl{10.1016/j.acha.2004.11.004}}.

\bibitem{Mhaskar2006weighted}
{\sc H.~N. Mhaskar}, {\em {Weighted quadrature formulas and approximation by
  zonal function networks on the sphere}}, J. Complexity, 22 (2006),
  pp.~348--370,
  \href{http://dx.doi.org/10.1016/j.jco.2005.10.003}{doi:\nolinkurl{10.1016/j.jco.2005.10.003}}.

\bibitem{mhaskar2001spherical}
{\sc H.~N. Mhaskar, F.~J. Narcowich, and J.~D. Ward}, {\em Spherical
  {M}arcinkiewicz--{Z}ygmund inequalities and positive quadrature}, Math.
  Comp., 70 (2001), pp.~1113--1130,
  \href{http://dx.doi.org/10.1090/S0025-5718-00-01240-0}{doi:\nolinkurl{10.1090/S0025-5718-00-01240-0}}.

\bibitem{MR4462410}
{\sc G.~Mont\'ufar and Y.~G. Wang}, {\em Distributed learning via filtered
  hyperinterpolation on manifolds}, Found. Comput. Math., 22 (2022),
  pp.~1219--1271,
  \href{http://dx.doi.org/10.1007/s10208-021-09529-5}{doi:\nolinkurl{10.1007/s10208-021-09529-5}}.

\bibitem{MR0199449}
{\sc C.~M\"{u}ller}, {\em {Spherical Harmonics}}, Springer Berlin, Heidelberg,
  1966.

\bibitem{MR1761902}
{\sc M.~Reimer}, {\em Hyperinterpolation on the sphere at the minimal
  projection order}, J. Approx. Theory, 104 (2000), pp.~272--286,
  \href{http://dx.doi.org/10.1006/jath.2000.3454}{doi:\nolinkurl{10.1006/jath.2000.3454}}.

\bibitem{reimer2002generalized}
{\sc M.~Reimer}, {\em Generalized hyperinterpolation on the sphere and the
  {N}ewma--{S}hapiro operators}, Constr. Approx., 18 (2002), pp.~183--204,
  \href{http://dx.doi.org/10.1007/s00365-001-0008-6}{doi:\nolinkurl{10.1007/s00365-001-0008-6}}.

\bibitem{sloan1995polynomial}
{\sc I.~H. Sloan}, {\em Polynomial interpolation and hyperinterpolation over
  general regions}, J. Approx. Theory, 83 (1995), pp.~238--254,
  \href{http://dx.doi.org/10.1006/jath.1995.1119}{doi:\nolinkurl{10.1006/jath.1995.1119}}.

\bibitem{MR1619076}
{\sc I.~H. Sloan}, {\em Interpolation and hyperinterpolation on the sphere}, in
  Multivariate Approximation, vol.~101 of Mathematical Research, Akademie
  Verlag, Berlin, 1997, pp.~255--268.

\bibitem{MR2875103}
{\sc I.~H. Sloan}, {\em Polynomial approximation on spheres---generalizing de
  la {V}all\'ee-{P}oussin}, Comput. Methods Appl. Math., 11 (2011),
  pp.~540--552,
  \href{http://dx.doi.org/10.2478/cmam-2011-0029}{doi:\nolinkurl{10.2478/cmam-2011-0029}}.

\bibitem{Sloan2018fourunate}
{\sc I.~H. Sloan}, {\em A fortunate scientific life}, in Contemporary
  Computational Mathematics - A Celebration of the 80th Birthday of Ian Sloan,
  Springer, Cham, 2018, pp.~xxi--xxviii.

\bibitem{MR494863}
{\sc I.~H. Sloan and W.~E. Smith}, {\em Product-integration with the
  {C}lenshaw--{C}urtis and related points. {C}onvergence properties}, Numer.
  Math., 30 (1978), pp.~415--428,
  \href{http://dx.doi.org/10.1007/BF01398509}{doi:\nolinkurl{10.1007/BF01398509}}.

\bibitem{MR577405}
{\sc I.~H. Sloan and W.~E. Smith}, {\em Product integration with the
  {C}lenshaw--{C}urtis points: implementation and error estimates}, Numer.
  Math., 34 (1980), pp.~387--401,
  \href{http://dx.doi.org/10.1007/BF01403676}{doi:\nolinkurl{10.1007/BF01403676}}.

\bibitem{MR650061}
{\sc I.~H. Sloan and W.~E. Smith}, {\em Properties of interpolatory product
  integration rules}, SIAM J. Numer. Anal., 19 (1982), pp.~427--442,
  \href{http://dx.doi.org/10.1137/0719027}{doi:\nolinkurl{10.1137/0719027}}.

\bibitem{zbMATH01421286}
{\sc I.~H. Sloan and R.~S. Womersley}, {\em The uniform error of
  hyperinterpolation on the sphere}, in Advances in Multivariate Approximation,
  vol.~107 of Mathematical Research, Wiley-VCH,Berlin, 1999, pp.~289--306.

\bibitem{sloan2012filtered}
{\sc I.~H. Sloan and R.~S. Womersley}, {\em Filtered hyperinterpolation: a
  constructive polynomial approximation on the sphere}, GEM Int. J. Geomath., 3
  (2012), pp.~95--117,
  \href{http://dx.doi.org/10.1007/s13137-011-0029-7}{doi:\nolinkurl{10.1007/s13137-011-0029-7}}.

\bibitem{MR4859492}
{\sc A.~Sommariva}, {\em Numerical cubature and hyperinterpolation over
  spherical polygons}, Appl. Math. Comput., 495 (2025), pp.~Paper No. 129335,
  12,
  \href{http://dx.doi.org/10.1016/j.amc.2025.129335}{doi:\nolinkurl{10.1016/j.amc.2025.129335}}.

\bibitem{sommariva2021numerical}
{\sc A.~Sommariva and M.~Vianello}, {\em Numerical hyperinterpolation over
  spherical triangles}, Math. Comput. Simulation, 190 (2021), pp.~15--22,
  \href{http://dx.doi.org/10.1016/j.matcom.2021.05.003}{doi:\nolinkurl{10.1016/j.matcom.2021.05.003}}.

\bibitem{MR0000077}
{\sc G.~Szeg\H{o}}, {\em Orthogonal Polynomials}, American Mathematical Society
  Colloquium Publications, Vol. 23, American Mathematical Society, New York,
  1939,
  \href{http://dx.doi.org/10.1090/coll/023}{doi:\nolinkurl{10.1090/coll/023}}.

\bibitem{trefethen2022exactness}
{\sc L.~N. Trefethen}, {\em Exactness of quadrature formulas}, SIAM Rev., 64
  (2022), pp.~132--150,
  \href{http://dx.doi.org/10.1137/20M1389522}{doi:\nolinkurl{10.1137/20M1389522}}.

\bibitem{MR3011254}
{\sc J.~Wade}, {\em On hyperinterpolation on the unit ball}, J. Math. Anal.
  Appl., 401 (2013), pp.~140--145,
  \href{http://dx.doi.org/10.1016/j.jmaa.2012.11.052}{doi:\nolinkurl{10.1016/j.jmaa.2012.11.052}}.

\bibitem{wang2014norm}
{\sc H.~Wang, K.~Wang, and X.~Wang}, {\em On the norm of the hyperinterpolation
  operator on the $d$-dimensional cube}, Comput. Math. Appl., 68 (2014),
  pp.~632--638,
  \href{http://dx.doi.org/10.1016/j.camwa.2014.07.009}{doi:\nolinkurl{10.1016/j.camwa.2014.07.009}}.

\bibitem{Wienert1990}
{\sc L.~Wienert}, {\em Die numerische Approximation von Randintegraloperatoren
  f\"{u}r die Helmholtzgleichung in $\mathbb{R}^3$}, PhD thesis, University of
  G\"{o}ttingen,, 1990.

\bibitem{MR1845243}
{\sc R.~S. Womersley and I.~H. Sloan}, {\em How good can polynomial
  interpolation on the sphere be?}, Adv. Comput. Math., 14 (2001),
  pp.~195--226,
  \href{http://dx.doi.org/10.1023/A:1016630227163}{doi:\nolinkurl{10.1023/A:1016630227163}}.

\bibitem{wu2024marcinkiewicz}
{\sc H.-N. Wu}, {\em Marcinkiewicz--{Z}ygmund inequalities for scattered data
  on polygons}, arXiv preprint arXiv:2411.16584,  (2024).

\bibitem{wu2023breaking}
{\sc H.-N. Wu and X.~Yuan}, {\em Breaking quadrature exactness: {A} spectral
  method for the {A}llen--{C}ahn equation on spheres}, arXiv preprint
  arXiv:2305.04820,  (2023).

\bibitem{MR1418495}
{\sc Y.~Xu}, {\em Lagrange interpolation on {C}hebyshev points of two
  variables}, J. Approx. Theory, 87 (1996), pp.~220--238,
  \href{http://dx.doi.org/10.1006/jath.1996.0102}{doi:\nolinkurl{10.1006/jath.1996.0102}}.

\bibitem{MR107776}
{\sc A.~Zygmund}, {\em Trigonometric Series. {V}ol {II}}, Cambridge University
  Press, Cambridge, 1959.

\end{thebibliography}
\end{document}